\documentclass[a4paper]{amsart}

\usepackage{amsmath,amssymb,amsthm}
\usepackage{fullpage}
\usepackage[colorlinks,linkcolor=blue,urlcolor=blue,citecolor=blue,destlabel=true]{hyperref}
\usepackage{mathtools}
\usepackage{hyperref}

\theoremstyle{plain}
\newtheorem{theorem}{Theorem}
\newtheorem*{theorem-mainIntro}{Theorem~\ref{thm:main}}
\newtheorem{lemma}{Lemma}[section]

\newtheorem{proposition}[lemma]{Proposition}

\theoremstyle{definition}
\newtheorem{definition}[lemma]{Definition}

\newtheorem{question}[lemma]{Question}
\newtheorem{remark}[lemma]{Remark}

\newcommand{\Q}{\mathbb{Q}}
\newcommand{\R}{\mathbb{R}}

\newcommand{\cA}{\mathcal{A}}
\newcommand{\cF}{\mathcal{F}}
\newcommand{\cI}{\mathcal{I}}
\newcommand{\cN}{\mathcal{N}}

\newcommand{\cU}{\mathcal{U}}
\newcommand{\cV}{\mathcal{V}}
\newcommand{\cVm}{\mathcal{V}^\mathrm{m}}

\newcommand{\vr}[2]{\mathrm{VR}(#1;#2)}
\newcommand{\vrleq}[2]{\mathrm{VR}_\leq(#1;#2)}

\newcommand{\cech}[2]{\mathrm{\check{C}}(#1;#2)}

\newcommand{\cechm}[2]{\mathrm{\check{C}}^\mathrm{m}(#1;#2)}

\newcommand{\vrm}[2]{\mathrm{VR}^\mathrm{m}(#1;#2)}
\newcommand{\vrmleq}[2]{\mathrm{VR}^\mathrm{m}_\leq(#1;#2)}

\newcommand{\Km}{K^\mathrm{m}}
\newcommand{\cPfin}[1]{\ensuremath{\mathcal{P}}^\mathrm{fin}(#1)}
\newcommand{\cpl}{\ensuremath{\mathrm{Cpl}}}

\newcommand{\diam}{\mathrm{diam}}
\newcommand{\supp}{\mathrm{supp}}

\newcommand{\MCP}{\mathrm{MCP}}

\newcommand{\im}{\mathrm{im}}

\newcommand{\wM}{\widetilde{M}}
\newcommand{\tc}{\tilde{c}}

\title[Vietoris thickenings and complexes have isomorphic homotopy groups]{Vietoris thickenings and complexes have \\ isomorphic homotopy groups}
\author{Henry Adams}
\address[HA]{Department of Mathematics, Colorado State University, Fort Collins, CO 80523, USA}
\email{henry.adams@colostate.edu}
\author{Florian Frick}
\address[FF]{Dept.\ Math.\ Sciences, Carnegie Mellon University, Pittsburgh, PA 15213, USA \newline \indent Inst. Math., Freie Universit\"at Berlin, Arnimallee 2, 14195 Berlin, Germany}
\email{frick@cmu.edu}
\author{\v{Z}iga Virk}
\address[\v ZV]{Faculty of Computer and Information Science, Ve\v cna pot 113, SI-1000 Ljubljana, Slovenia}
\email{ziga.virk@fri.uni-lj.si}
\date{\today}


\begin{document}

\maketitle

\begin{abstract}
\small
We study the relationship between metric thickenings and simplicial complexes associated to coverings of metric spaces.
Let $\cU$ be a cover of a separable metric space $X$ by open sets with a uniform diameter bound.
The Vietoris complex~$\cV(\cU)$ contains all simplices with vertex set contained in some $U \in \cU$, and the Vietoris metric thickening~$\cVm(\cU)$ is the space of probability measures with support in some $U \in \cU$, equipped with an optimal transport metric.
We show that~$\cVm(\cU)$ and~$\cV(\cU)$ have isomorphic homotopy groups in all dimensions.
In particular, by choosing the cover $\cU$ appropriately, we get isomorphisms between the homotopy groups of Vietoris--Rips metric thickenings and simplicial complexes $\pi_k(\vrm{X}{r})\cong \pi_k(\vr{X}{r})$ for all integers $k\ge 0$, where both spaces are defined using the convention ``diameter $< r$'' (instead of $\le r$).
Similarly, we get isomorphisms between the homotopy groups of \v{C}ech metric thickenings and simplicial complexes $\pi_k(\cechm{X}{r})\cong \pi_k(\cech{X}{r})$ for all integers $k\ge 0$, where both spaces are defined using open balls (instead of closed balls).
\end{abstract}

%\setcounter{tocdepth}{1}
%\tableofcontents

\section{Introduction}

Given only a partial sampling $X$ from an unknown metric space $M$, how can one recover properties of the entire metric space $M$?
Questions of this kind frequently arise in topological data analysis, where one would like to understand the ``shape'' of a dataset~$X$, which is sometimes defined using the shape of the larger underlying metric space $M$ from which the data set $X$ was sampled~\cite{Carlsson2009}.

Vietoris--Rips and \v{C}ech simplicial complexes are useful tools for questions along these lines.
The idea is to ``thicken'' the dataset $X$, obtaining a Vietoris--Rips simplicial complex or a \v{C}ech simplicial complex, each of which have $X$ as their vertex set.
These constructions depend on the choice of a real-valued scale parameter $r\ge 0$ that quantifies the degree to which $X$ has been thickened: more and more simplices are added as $r$ increases.
Indeed, the \v{C}ech complex $\cech{X}{r}$ is the nerve of the open balls of radius $r$ centered at points in $X$, and so when the nerve theorem applies, the \v{C}ech complex is homotopy equivalent to the union of these balls.
The Vietoris--Rips complex instead contains as its simplices all finite subsets of $X$ of diameter less than $r$.
The Vietoris--Rips complex is closely related to the \v{C}ech complex, but is easier to compute since as a \emph{clique} or \emph{flag} complex it is completely determined by its underlying 1-skeleton graph.

Let us explain how Vietoris--Rips and \v{C}ech complexes of a partial sampling $X$ help recover properties of the entire metric space~$M$.
When $M$ is a sufficiently nice manifold, when $X$ is sufficiently close to~$M$, and when the scale parameter $r$ is chosen carefully, results by Latschev~\cite{Latschev2001}, Hausmann~\cite{Hausmann1995,virk2021rips,virk2021counter}, and Niyogi, Smale, and Weinberger~\cite{niyogi2008finding} say that the Vietoris--Rips simplicial complex $\vr{X}{r}$ and \v{C}ech simplicial complex $\cech{X}{r}$ have the same homotopy type as $M$.
Though these results apply in slightly different settings, they all share related assumptions that the scale parameter $r$ needs to be small compared to the curvature of the manifold~$M$.
Unfortunately, since $M$ is unknown, so is its curvature!
One is therefore left without clear guidance as to how to choose the scale parameter~$r$.
Fortunately, persistent homology~\cite{edelsbrunner2000topological,EdelsbrunnerHarer,zomorodian2005computing} enables some tools for avoiding this choice of scale $r$.
Indeed, the idea of persistent homology is to allow the scale $r$ to vary from small to large, and to track the topological changes of $\vr{X}{r}$ or $\cech{X}{r}$ as $r$ increases.
An important result is the stability of persistent homology~\cite{cohen2007stability,chazal2009gromov,ChazalDeSilvaOudot2014}, which in this context implies that if $X$ is close to $M$, then the persistent homology of $X$ is close to the persistent homology of $M$.
Furthermore, as more and more data points are sampled from a manifold $M$, then the persistent homology of the growing dataset converges to the persistent homology of the manifold $M$.

The question thus naturally arises: what is the persistent homology of Vietoris--Rips and \v{C}ech complexes of manifolds, as the scale parameter $r$ increases?
The Vietoris--Rips complexes of the circle obtain the homotopy types of the circle, the 3-sphere, the 5-sphere, the 7-sphere, etc., as the scale $r$ increases, until finally the complex is contractible~\cite{AA-VRS1,Adamaszek2013,AAFPP-J}.
These odd-sphere homotopy types have consequences for the persistent homology of spaces containing geodesic loops~\cite{virk2021footprints}.
Much less is known about the persistent homology of Vietoris--Rips complexes of $n$-spheres~\cite{AAF,ABF2,lim2020vietoris,zaremsky2019}.
The $1$-dimensional persistent homology of Vietoris--Rips and \v{C}ech complexes of geodesic spaces is completely classified by~\cite{virk2017approximations,virk20201,gasparovic2018complete}.
It is possible to put bounds on the length of higher-dimensional persistent homology bars by relating Vietoris--Rips complexes to the spread of a metric space and Gromov's filling radius~\cite{lim2020vietoris}, or instead to notions from geometric topology~\cite{adams2021geometric}.
However, very little is known in general about the persistent homology of Vietoris--Rips complexes $\vr{M}{r}$ or \v{C}ech complexes $\cech{M}{r}$ of a manifold $M$.

One reason why so little is known is that the topology of a Vietoris--Rips or \v{C}ech simplicial complex is at times difficult to work with.
For example, if $M$ is a manifold of dimension at least one, then the inclusion map $M\hookrightarrow \vr{M}{r}$ or $M\hookrightarrow \cech{M}{r}$ is not continuous since the vertex set of any simplicial complex is equipped with the discrete topology.
This situation is improved by considering instead the \emph{metric thickenings} $\vrm{M}{r}$ and $\cechm{M}{r}$~\cite{AAF}, which are in natural bijection with the geometric realizations of $\vr{M}{r}$ and $\cech{M}{r}$, but which are equipped with a more natural topology (indeed a metric) arising from ideas in optimal transport.
The inclusion maps $M\hookrightarrow \vrm{M}{r}$ or $M\hookrightarrow \cechm{M}{r}$ are now continuous, and in fact isometric embeddings onto their images.
It has recently been proven in~\cite{AMMW,MoyMasters} that the Vietoris--Rips and \v{C}ech metric thickenings have the same persistent homology barcodes as the corresponding simplicial complexes,\footnote{So long as one ignores whether an endpoint of a bar is open or closed.} enabling one to use either simplicial or metric techniques.
At times, the metric thickening allows one to go further; for example we can describe the first new homotopy type that appears in Vietoris--Rips metric thickenings of $n$-spheres for all $n$, even though the first new homotopy type for Vietoris--Rips simplicial complexes of $n$-spheres is only known for $n\le 2$~\cite{AA-VRS1,lim2020vietoris,katz1989diameter,katz1991neighborhoods}.
An additional motivating reason to better understand the relation between Vietoris--Rips simplicial complexes and metric thickenings is as follows.
In~\cite{lim2020vietoris}, the authors prove that the Vietoris--Rips filtration $\vr{X}{-}$ is isomorphic to the tubular neighborhoods thickening of $X$ inside an ambient injective metric space.
Hence, a better understanding of the relationship between $\vrm{X}{r}$ and $\vr{X}{r}$ may point to a deeper connection between Wasserstein spaces and injective (or hyperconvex) metric spaces.

In this paper we will define the Vietoris--Rips complex $\vr{X}{r}$ and metric thickening $\vrm{X}{r}$ using the inequality $<r$, though there are analogous versions $\vrleq{X}{r}$ and $\vrmleq{X}{r}$ defined using the inequality $\le r$.
It is known that the homotopy types of $\vrleq{X}{r}$ and $\vrmleq{X}{r}$ need not be the same.\footnote{
For example, when $r=0$, then $\vrleq{X}{0}$ is $X$ equipped with the discrete topology, whereas $\vrmleq{X}{0}$ is the metric space $X$ equipped with its standard topology.
A less trivial example is that if $S^1$ is the geodesic circle of circumference $2\pi$, then $\vrleq{S^1}{\frac{2\pi}{3}}\simeq \bigvee^\infty S^2$ is an uncountably infinite wedge sum of 2-dimensional spheres~\cite{AA-VRS1}, whereas $\vrmleq{S^1}{\frac{2\pi}{3}}\simeq S^3$ obtains the expected homotopy type of a 3-sphere~\cite{AAF,ABF}.
We say ``expected'' since we do have $\vrleq{S^1}{r}\simeq S^3$ for all $\frac{2\pi}{3}<r<\frac{4\pi}{5}$.
This entire footnote has analogues for \v{C}ech complexes and thickenings, as well.}
Nevertheless, it is reasonable to conjecture that if we return to using the inequality $<r$, then we have homotopy equivalences $\vr{X}{r}\simeq\vrm{X}{r}$ and $\cech{X}{r}\simeq\cechm{X}{r}$ for all separable metric spaces $X$ and scales $r\ge0$; see~\cite[Remark~3.3]{AAF}
%~\cite[Question~(1) in Section~9]{AMMW}
and Question~\ref{ques:homotopy-equiv} in this paper.
This conjecture is known to be true if $X$ is a discrete metric space~\cite[Proposition~6.6]{AAF},
or if $X$ is a compact manifold and the scale $r$ is sufficiently small compared to the manifold (in which case $\vr{X}{r}$, $\vrm{X}{r}$, $\cech{X}{r}$, and $\cechm{X}{r}$ are each homotopy equivalent to the manifold~\cite{Hausmann1995,AAF,AM}), or if $X=S^1$ is the circle and $r$ is arbitrary by recent work of Moy~\cite{moyVRmS1}.
The portion of the conjecture that remains open is when $X$ is not discrete and the scale $r$ is arbitrary, and for example the conjecture is open when $X$ is any manifold of dimension $\ge2$.
In this paper we provide positive results in the direction of this conjecture.
In our main result (Theorem~\ref{thm:main}), we prove that if $X$ is a separable metric space and if $r\ge0$, then $\vrm{X}{r}$ and $\vr{X}{r}$ have isomorphic homotopy groups, and similarly $\cechm{X}{r}$ and $\cech{X}{r}$ have isomorphic homotopy groups.
In other words, we prove that $\pi_k(\vrm{X}{r}) \cong \pi_k(\vr{X}{r})$ and $\pi_k(\cechm{X}{r}) \cong \pi_k(\cech{X}{r})$ for all integers $k\ge 0$.
Our main result is in fact more general:

\begin{theorem-mainIntro}
If $\cU$ is a uniformly-bounded open cover of a separable metric space $X$, then the Vietoris thickening $\cVm(\cU)$ and the Vietoris complex $\cV(\cU)$ have isomorphic homotopy groups $\pi_k(\cVm(\cU))\cong \pi_k(\cV(\cU))$ in all dimensions $k\ge 0$.
\end{theorem-mainIntro}

Here, the \emph{Vietoris simplicial complex} $\cV(\cU)$ has $X$ as its vertex set, and a finite subset $\sigma\subseteq X$ as a simplex when there exists some $U\in\cU$ with $\sigma\subseteq U$.
The \emph{Vietoris metric thickening} $\cVm(\cU)$ of all finitely-supported probability measures supported in some $U\in\cU$ is equipped with a different topology, which furthermore is induced by an optimal transport metric extending the metric on $X$; see Section~\ref{sec:prelim}.

By choosing $\cU$ to be the open cover of $X$ by all sets of diameter $<r$, we obtain the previously mentioned isomorphisms between the homotopy groups of $\cVm(\cU)=\vrm{X}{r}$ and $\cV(\cU)=\vr{X}{r}$.
Similarly, by choosing $\cU$ to be the open cover of $X$ by all balls of radius $<r$, we obtain the previously mentioned isomorphisms between the homotopy groups of $\cVm(\cU)=\cechm{X}{r}$ and $\cV(\cU)=\cech{X}{r}$.

The organization of our paper is as follows.
In Section~\ref{sec:prelim} we describe necessary preliminaries and we set notation.
In Section~\ref{sec:cover-good} we describe a cover of $\cVm(\cU)$ that is good but not open---one can think of the fact that this cover is not open as an obstacle towards attempted proofs of the (still-open) conjecture that $\cVm(\cU)$ and $\cV(\cU)$ have the same homotopy type.
In Section~\ref{sec:cover-open} we modify this good cover that is not open in order to obtain an open cover that is only ``good up to level $n$.''
This open cover is what we need in Section~\ref{sec:main} to prove Theorem~\ref{thm:main} that $\cVm(\cU)$ and $\cV(\cU)$ have isomorphic homotopy groups.
Indeed, the main tool we use in our proof is different versions of the nerve theorem, including versions by Nagorko~\cite{nagorko2007carrier} and the third author~\cite{virk2021rips}.
The proof of Theorem~\ref{thm:main} relies on the fact that $\cVm(\cU)$ is locally contractible, which we prove in Theorem~\ref{thm:locally-contractible} in Section~\ref{sec:locally-contractible}.
We conclude and pose open questions in Section~\ref{sec:conclusion}.

\section{Preliminaries}
\label{sec:prelim}

Let us define the main objects of study: the nerve and Vietoris complexes of an open cover, Vietoris--Rips and \v{C}ech simplicial complexes, and metric thickenings.

\subsection*{Topological spaces and metric spaces}

Let $X=(X,d)$ be a metric space.
For $x\in X$ and $r\ge 0$, we let $B_X(x;r):=\{x'\in X~|~d(x,x')<r\}$ be the open ball of radius $r$ about $x$ in $X$.
This ball is empty if $r=0$.
The \emph{diameter} of a subset $\sigma\subseteq X$ is defined as $\diam(\sigma)=\sup_{x,x'\in \sigma}d(x,x')$.
For two nonempty subsets $A,B\subseteq X$, we define $d(A,B):=\inf_{a\in A,b\in B} d(a,b)$ to be the distance between the two sets.

If $(X,d_X)$ and $(Y,d_Y)$ are two metric spaces, if $f\colon X\to Y$ is a map between them, and if $L$ is a nonnegative real number, then we say that $f$ is \emph{$L$-Lipschitz} if $d_Y(f(x),f(x'))\le L\,d_X(x,x')$ for all $x,x'\in X$.

Let $X$ be a metric space.
For $U\subseteq X$, let $U^C=X\setminus U$ denote the complement of $U$ in $X$.
If $U$ is open in $X$, then there exists a 1-Lipschitz map $\phi\colon X \to [0,1]$ with $\phi^{-1}(0)=U^C$.
For example, if $U^C\neq\emptyset$ then one can let $\phi(x)=\min(d(x,U^C),1)$, and if $U^C = \emptyset$ then one can simply let $\phi$ be the constant function~$1$.

For $Y$ a topological space, we let $\pi_k(Y)$ denote its $k$-th homotopy group.
For two topological spaces $Y$ and $Y'$, we write ``$\pi_k(Y)\cong\pi_k(Y')$ for all integers $k\ge 0$'' to mean that the spaces $Y$ and $Y'$ have the same number $m$ of path-connected components, and that there exist points $y_1,\ldots,y_m$ and $y'_1,\ldots,y'_m$ from distinct components in $Y$ and $Y'$ along with a bijection $h\colon\{y_1,\ldots,y_m\}\to\{y'_1,\ldots,y'_m\}$ such that $\pi_k(Y,y_i)\cong\pi_k(Y',h(y_i))$ for all integers $k\ge 0$ and for all $1\le i\le m$.
We say that a space $Y$ is \emph{$n$-connected} if $\pi_k(Y)$ is the trivial group for all $0\le k\le n$.

\subsection*{Nerve and Vietoris complex of a cover}

Let $\cU$ be a cover of the metric space $X$ by nonempty sets.
Therefore, each point $x\in X$ satisfies $x\in U$ for at least one set $U\in \cU$.
We may write $\cU=\{U_\alpha\}_{\alpha\in \cI}$, where the index set $\cI$ could be finite, countably infinite, or uncountably infinite.
We allow the possibility that $U_\alpha=U_{\alpha'}$ for $\alpha\neq \alpha'$, but we require $U_\alpha\neq\emptyset$ for all $\alpha\in\cI$.
We say that the cover $\cU$ is \emph{open} if each $U\in\cU$ is an open set in $X$.
We say that the cover $\cU$ is \emph{uniformly bounded} if there exists some constant $D<\infty$ such that $\diam(U)<D$ for each $U\in\cU$.

For example, if $\cU$ is the open cover of $X$ consisting of all open sets of diameter less than $r$, then $\cU$ is $r$-bounded.
Similarly, if $\cU$ is the cover of $X$ by open balls of radius $r$, then $\cU$ is $2r$-bounded.

The \emph{nerve simplicial complex} $\cN(\cU)$ has $\cI$ as its vertex set, and has a finite subset $\sigma\subseteq\cI$ as a simplex if $\cap_{\alpha\in \sigma}U_\alpha\neq\emptyset$~\cite{Borsuk1948}.
The \emph{Vietoris simplicial complex} $\cV(\cU)$ has $X$ as its vertex set, and has a finite subset $\sigma\subseteq X$ as a simplex if there exists some $U\in\cU$ with $\sigma\subseteq U$.
By Dowker duality~\cite{dowker1952topology,zeeman1962dihomology,virk2021rips}, the complexes $\cN(\cU)$ and $\cV(\cU)$ are homotopy equivalent.
See Figure~\ref{fig:nerve-Vietoris}.

\begin{figure}[h]
\centering
\includegraphics[width=4in]{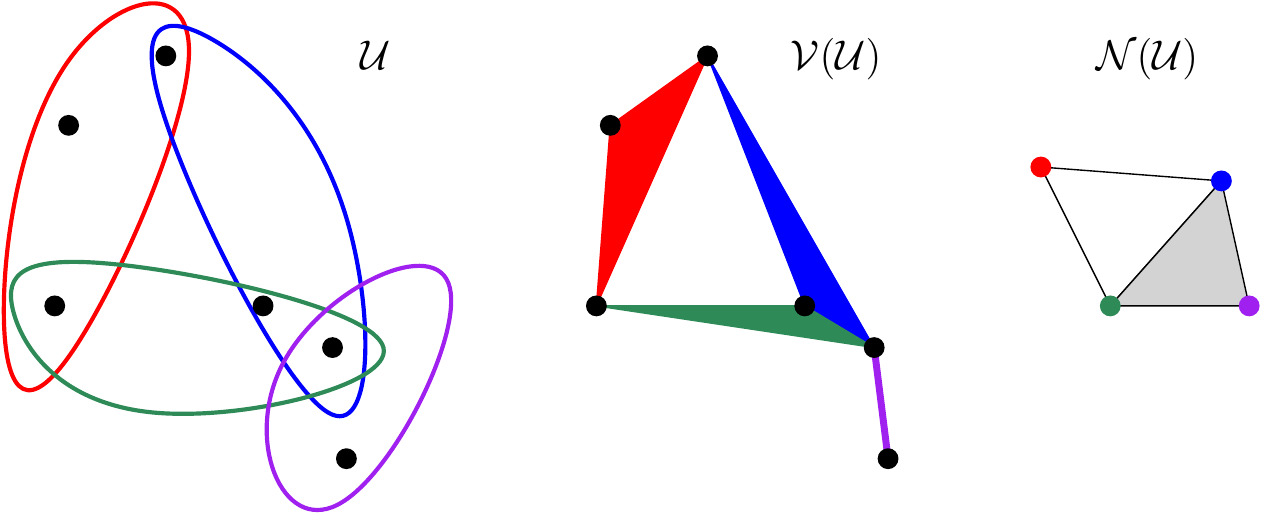}
\caption{This figure is taken from~\cite[Figure~5.15]{VirkBook}.
(Left) A cover $\cU$ of six points by
four colored sets, (center) its Vietoris
complex $\cV(\cU)$, and (right) its nerve complex $\cN(\cU)$.}
\label{fig:nerve-Vietoris}
\end{figure}

For $\sigma\subseteq \cI$, we let $U_\sigma:=\cap_{\alpha\in \sigma} U_\alpha$ denote the intersection of the sets from $\sigma$.
If $U_\alpha$ is contractible for each $\alpha\in\cI$, and if $U_\sigma$ is either contractible or empty for each $\sigma\subseteq \cI$, then we say that the cover $\cU$ is a \emph{good cover}.
The nerve theorem provides relatively mild point-set topology assumptions\footnote{To see that point-set topology assumptions are needed, consider a cover of a connected space $X$ by two disjoint sets.}
on $X$ and on $\cU$ so that if $\cU$ is a good cover of $X$, then the nerve $\cN(\cU)$ is homotopy equivalent to the space $X$.
This theorem applies, for example, if $\cU$ is an open cover of a paracompact space $X$, or if $\cU$ is a cover of a simplicial complex by subcomplexes~\cite{Borsuk1948,Dieck,weil1952theoremes,Hatcher}.

If $U_\alpha$ is contractible for each $\alpha\in\cI$, and if $U_\sigma$ is either contractible or empty for each $\sigma\subseteq \cI$ of size $|\sigma|\le n$, then we say that the cover $\cU$ is a \emph{good cover up to level $n$}.
There are generalized versions of the nerve lemma which, when $\cU$ is only a good cover up to a certain level, can still imply that the nerve $\cN(\cU)$ and $X$ have matching homotopy groups up to a certain dimension~\cite{bjorner2003nerves}.
We will use such a result by Nagorko~\cite{nagorko2007carrier}.

\subsection*{Vietoris--Rips and \v{C}ech simplicial complexes}

For $X$ a metric space and for $r>0$, the \emph{Vietoris--Rips simplicial complex} $\vr{X}{r}$ contains $X$ as its vertex set, and a finite subset $\sigma \subseteq X$ as a simplex if $\diam(\sigma)< r$.
If $\cU$ is chosen to be the open cover of $X$ consisting of all open sets of diameter less than $r$, then the Vietoris--Rips complex is the Vietoris complex of this cover, namely $\vr{X}{r}=\cV(\cU)$~\cite{virk2021rips}.
By convention, $\vr{X}{0}$ is the empty simplicial complex (with no vertices).

For $x\in X$ and $r\ge 0$, we let $B(x,r):=\{x'\in X~|~d(x,x')<r\}$ denote the open ball of radius $r$ centered at the point $x$.
For $r>0$, the \emph{\v{C}ech simplicial complex} $\cech{X}{r}$ contains $X$ as its vertex set, and a finite subset $\sigma \subseteq X$ as a simplex if $\cap_{x\in \sigma}B(x,r)\neq\emptyset$.
Equivalently, if $\cU=\{B(x,r)\}_{x\in X}$ is defined to be the cover of $X$ by open balls of radius $r$, then the \v{C}ech complex is equal to the nerve complex of this cover, namely $\cech{X}{r}=\cN(\cU)$.
Interestingly, one can also see that the \v{C}ech complex is the Vietoris complex of this same cover, namely $\cech{X}{r}=\cV(\cU)$.
This is true since balls of radius $r$ in $X$ intersect at a common point if and only if their centers are all contained in a ball of radius $r$ about the intersection point.
Note the equality $\cN(\cU)=\cV(\cU)$ here is possible only since the vertex sets of these two simplicial complexes agree, which is the case since $\cU=\{B(x,r)\}_{x\in X}$ is a cover of $X$ whose sets are also indexed by the points in $X$.

We have seen that one can realize both Vietoris--Rips complexes and \v{C}ech complexes, two of the most popular simplical complexes in applied topology, as Vietoris complexes of different covers.
Therefore, Vietoris complexes will be the main simplicial complexes of interest in the rest of this paper.

\subsection*{Optimal transport}

For $X$ a metric space, let $\cPfin{X}$ be the set of all probability measures on $X$ with finite support.
In other words, each measure $\mu\in\cPfin{X}$ can be written as $\mu=\sum_{i=0}^k a_i\delta_{x_i}$ with $k\ge 0$, $a_i>0$, $\sum a_i=1$, and $x_i\in X$ for all $i$.
We define the \emph{support} of this finitely supported measure to be $\supp(\mu)=\{x_0,\ldots,x_k\}$.
Here $\delta_{x}$ is the Dirac probability measure with unit mass at the point $x\in X$.
A \emph{coupling} between two measures $\mu,\nu\in\cPfin{X}$ is a probability measure $\gamma \in \cPfin{X\times X}$ whose marginals on the first and second factors are $\mu$ and $\nu$.
Note that $\gamma$ has finite support since $\mu$ and $\nu$ do.
Let $\cpl(\mu,\nu)$ denote the set of all couplings between these two measures.
Then, for a real number $q\ge 1$, the \emph{$q$-Wasserstein distance} between $\mu$ and $\nu$ is
\[d_{W,q}(\mu,\nu):=\inf_{\gamma\in\cpl(\mu,\nu)}\left(\int_{X\times X}d(x,y)^q\gamma(dx \times dy)\right)^{1/q}.\]
It is easy to see that this infimium is realized since $\gamma$ has finite support.
In particular, we have $d_{W,q}(\delta_x,\delta_y)=d(x,y)$ for any $x,y\in X$.
The Wasserstein metric on the space of probability Radon measures has many names: the Kantorovich, optimal transport, or earth mover's metric~\cite{vershik2013long,villani2003topics,villani2008optimal}.
It is known that in a variety of different contexts, the $q$-Wasserstein metric induces the same (weak) topology for any $q \in [1, \infty)$; see~\cite{bogachev2018weak} and~\cite[Appendix~A]{AMMW}.

We will frequently use the following lemma in order to construct continuous homotopies.

\begin{lemma}
\label{lem:homotopy1}
Suppose $A \subseteq \cPfin{X}$.
If $f \colon A \to \cPfin{X}$ is continuous, then so is the homotopy
$H \colon A \times [0, 1] \to \cPfin{X}$ given by $H(\mu, t) = (1 - t)\mu + t f(\mu)$.
\end{lemma}
This follows, for example, from~\cite[Proposition~2.4]{AMMW} or~\cite[Lemma~3.9]{AAF}; see also~\cite{bogachev2018weak}.

For convenience, for the remainder of this paper we will use the $1$-Wasserstein metric, which we denote simply by $d_W$.
We will also use finitely-supported measures.
If $\mu=\sum_{i=0}^k a_i\delta_{x_i}$ and $\nu=\sum_{j=0}^l b_j\delta_{y_j}$
are two finitely-supported probability measures, then a coupling between them is a collection of nonnegative real numbers $\{c_{i,j}\}$ for $1\le i\le k$ and $1\le j\le l$ such that $\sum_j c_{i,j}=a_i$ for all $i$, and such that $\sum_i c_{i,j}=b_j$ for all $j$; it follows that $\sum_{i,j}c_{i,j}=1$.
The cost of such a coupling is $\sum_{i,j}c_{i,j}d_X(x_i,y_j)$, and the 1-Wasserstein distance is the infimal cost over all possible couplings.
In our setting of finitely-supported measures, an infimal coupling is always attained.

\subsection*{Partial couplings}

We will need the following lemma related to partial couplings.
A \emph{partial coupling} between two measures $\mu = \sum_i a_i \delta_{x_i}$ and $\nu= \sum_j b_j \delta_{y_j}$ is collection $\{c_{i,j}\}$ of nonnegative real numbers such that $\sum_j c_{i,j} \leq a_i$ and $\sum_i c_{i,j} \leq b_j$.
It gives incomplete information about how the mass is transported from $\mu$ to $\nu$, but any partial coupling from $\mu$ to $\nu$ can be completed to a coupling from $\mu$ to $\nu$.
That is, given a partial coupling $\{c_{i,j}\}$, there exists a coupling $\{\tc_{i,j}\}$ such that $c_{i,j} \leq \tc_{i,j}$ for all $i,j$.
Indeed, we can simply extend $\{c_{i,j}\}$ by using the product measure on any unmatched mass.
To be explicit, we define
\[\tc_{i,j}:=c_{i,j}+\frac{(a_i-\sum_j c_{i,j})(b_j-\sum_i c_{i,j})}{1-c},\quad\text{where}\quad c:=\sum_{i,j}c_{i,j}.\]
Indeed, we then have
\begin{align*}
\sum_j\tc_{i,j} &= \sum_j \left(c_{i,j}+\frac{(a_i-\sum_j c_{i,j})(b_j-\sum_i c_{i,j})}{1-c} \right) \\
&= \sum_jc_{i,j} + \frac{a_i-\sum_j c_{i,j}}{1-c} \sum_j \left(b_j-\sum_i c_{i,j}\right) \\
&= \sum_jc_{i,j} + \frac{a_i-\sum_j c_{i,j}}{1-c}\left(1-c\right) \\
&= a_i.
\end{align*}
A similar computation shows $\sum_i \tc_{i,j}=b_j$, and therefore $\{c_{i,j}\}$ is a coupling between $\mu$ and $\nu$.

The following lemma will allow us to bound the Wasserstein distance between two measures, even if we only construct a partial coupling between them.

\begin{lemma}[Partial coupling lemma]
\label{lem:partial-transport}
If $\{c_{i,j}\}$ is a partial coupling between $\mu=\sum_i a_i\delta_{x_i}$ and $\nu=\sum_j b_j\delta_{y_j}$, then
\[d_W(\mu,\nu) \le \sum_{i,j}c_{i,j}d_X(x_i,y_j) + \left(1-\sum_{i,j}c_{i,j}\right)\diam\left(\supp(\mu)\cup\supp(\nu)\right).\]
\end{lemma}

\begin{proof}
Choose any coupling $\{\tc_{i,j}\}$ between $\mu$ and $\nu$ with $c_{i,j} \leq \tc_{i,j}$; this is possible (for example) by extending $\{c_{i,j}\}$ using the product measure on any unmatched mass.
We have
\begin{align*}
d_W(\mu,\nu) &\le \sum_{i,j}\tc_{i,j}d_X(x_i,y_j) \\
&= \sum_{i,j}c_{i,j}d_X(x_i,y_j) + \sum_{i,j}(\tc_{i,j}-c_{i,j})d_X(x_i,y_j) \\
&\le \sum_{i,j}c_{i,j}d_X(x_i,y_j) + \left(1-\sum_{i,j}c_{i,j}\right)\diam\left(\supp(\mu)\cup\supp(\nu)\right).
\end{align*}
\end{proof}

\subsection*{Metric thickenings}

Let $X$ be a metric space, and let $K$ be a simplicial complex with vertex set $X$.
The \emph{metric thickening} $\Km$ is defined in~\cite{AAF} as the space of all probability measures that are supported on the vertex set of some simplex in~$K$, equipped with a Wasserstein metric.
More explicitly, $\Km$ is the metric space
\[\Km:=\left\{\mu\in\cPfin{X}~|~\supp(\mu)\in K\right\},\]
%equipped with the $q$-Wasserstein metric for some $1\le q< \infty$.
%Any such $q$ induces the same topology, and in this paper we use $q=1$ for convenience.
equipped with the $1$-Wasserstein metric.\footnote{As any $q$-Wasserstein metric for $1\le q< \infty$ induces the same topology, we make the choice $q=1$ for convenience.} 
We note that there is a natural bijection between the geometric realization of $K$ and $\Km$, obtained by assigning a point in the geometric realization of $K$ associated to the simplex $\{x_0,\ldots,x_k\}$ with barycentric coordinates $(a_0,\ldots,a_k)$ to the measure $\sum_{i=0}^k a_i\delta_{x_i}\in\Km$.
By~\cite[Proposition~6.1]{AAF} the map $K\to\Km$ given by this bijection is continuous, but the inverse map $\Km\to K$ may be discontinuous when $X$ is infinite.
When the simplicial complex $K$ is of the form $\cV(\cU)$, $\vr{X}{r}$, or $\cech{X}{r}$, then we denote $\Km$ by $\cVm(\cU)$, $\vrm{X}{r}$, or $\cechm{X}{r}$, respectively.

\section{A good cover that is not open}
\label{sec:cover-good}

We describe how to use a cover $\cU$ of a metric space $X$ (that is not necessarily a good cover) to build a good cover of the Vietoris metric thickening $\cVm(\cU)$.

For $Y\subseteq X$, let $M_Y\subseteq\cPfin{X}$ be the set of all finitely supported probability measures with support contained in $Y$.
If $Y$ is empty then so is $M_Y$.
If $Y$ is nonempty then $M_Y\simeq \ast$ is contractible.
Indeed, choose any $x\in Y$ and define the deformation retraction $H\colon M_Y\times [0,1] \to M_Y$, with $H(\cdot,0)$ the identity map on $M_Y$ and with $H(\cdot,1)$ the constant map to $\delta_x$, via $H(\sum a_i \delta_{x_i},t) = (1-t)\sum a_i \delta_{x_i} + t \delta_x$.

Let $\cU=\{U_\alpha\}_{\alpha\in\cI}$ be a cover of the metric space $X$.
Note that 
$M_\cU := \{M_U\}_{U \in \cU} = \{M_{U_\alpha}\}_{\alpha\in\cI}$
is a cover of the Vietoris metric thickening $\cVm(\cU)$.
For $\alpha\in \cI$ we have that $M_{U_\alpha}$ is contractible.
For $\sigma\subseteq \cI$ we have $\cap_{\alpha\in \sigma}M_{U_\alpha}=M_{U_\sigma}$.
Note that $M_{U_\sigma}$ is contractible if $U_\sigma\neq\emptyset$, i.e.\ if $\sigma\in\cN(\cU)$, and $M_{U_\sigma}$ is empty if $U_\sigma=\emptyset$, i.e.\ if $\sigma\notin\cN(\cU)$.
Therefore $M_\cU$ is a good cover of $\cVm(\cU)$, and furthermore $\cN(M_\cU)=\cN(\cU)$.
Applying the homotopy equivalence from~\cite[Proposition~3.7]{virk2021rips}, we get that $\cN(M_\cU)=\cN(\cU)\simeq \cV(\cU)$.
In summary, we have constructed a good cover of $\cVm(\cU)$ whose nerve is homotopy equivalent to the Vietoris complex $\cV(\cU)$.

Is there some version of the nerve theorem that we could apply to the good cover $M_\cU$?
\emph{If there were}, then we would be able to conclude that $\cVm(\cU)$ is homotopy equivalent to the nerve $\cN(M_\cU)$, which is homotopy equivalent $\cV(\cU)$.
Unfortunately, we don't know if the cover $M_\cU$ satisfies the point-set topology conditions needed to apply any version of the nerve theorem that we are aware of.

In particular, even if $\cU$ is an open cover of $X$, the cover $M_\cU$ of $\cVm(\cU)$ is not an open cover in general.\footnote{For example, if $x,y\in X$ and $U,U'\in\cU$ satisfy $x\in U\cap U'$ and $y\in U'\setminus U$, then $\delta_x\in M_U$.
Any open ball in $B_{\cVm(\cU)}(\delta_x;\varepsilon)$ contains points of the form $(1-\varepsilon')\delta_x + \varepsilon'\delta_y\notin M_U$ for $\varepsilon'>0$ sufficiently small.
This shows that $M_U$ is not open.}
One can think of the fact that this cover is not open as an obstacle towards attempted proofs of the (still-open) conjecture that $\cVm(\cU)$ and $\cV(\cU)$ have the same homotopy type; see Question~\ref{ques:homotopy-equiv}.

The goal of the remainder of the paper, in some sense, is to overcome the obstacle that the cover $M_\cU$ is not open.
We will construct a modified cover $\wM_\cU$ of $\cVm(\cU)$ that is open, while maintaining as many nice properties as possible.
In particular, though $\wM_\cU$ will not be a good cover in general, it will be good up to level $n$ for some fixed but arbitrary integer $n$.
This will be enough for us to prove that all of the homotopy groups of $\cVm(\cU)$ and $\cV(\cU)$ agree.

\section{An open cover, good up to level $n$}
\label{sec:cover-open}

Let $\cU$ be a uniformly bounded open cover of a metric space $X$. 
In this section we show how to produce a cover $\wM_\cU$ of $\cVm(\cU)$ that is open and good up to level $n$.
We begin with some preliminaries.

\begin{definition}
Let $X$ be a metric space, let $U\subseteq X$, and let $0<p<1$.
A subset $A \subseteq \cPfin{X}$ has the \emph{mass concentration property} $\MCP(p,U)$ if for each $\mu\in A$, more than $p$ of its mass is contained in $U$, i.e.\ $\mu(U)>p$.
\end{definition}

\begin{definition}
\label{def:pumping-convex}
Let $X$ be a metric space and $U\subseteq X$.
We say that $A \subseteq \cPfin{X}$ is \emph{$U$-pumping convex} if for any $\mu \in A$, any $\nu\in\cPfin{X}$ with $\supp(\nu)\subseteq \supp(\mu)\cap U$, and any $t\in [0,1]$, we have that $(1-t)\mu+t\nu\in A$.
\end{definition}

Note that an intersection of $U$-pumping convex sets is $U$-pumping convex, and a  union of $U$-pumping convex sets is $U$-pumping convex.

The following pumping lemma shows how to continuously deform a measure that has some of its mass in $U$ to instead have all of its mass in $U$; see Figure~\ref{fig:pumping}.

\begin{figure}[h]
\centering
\includegraphics[width=4in]{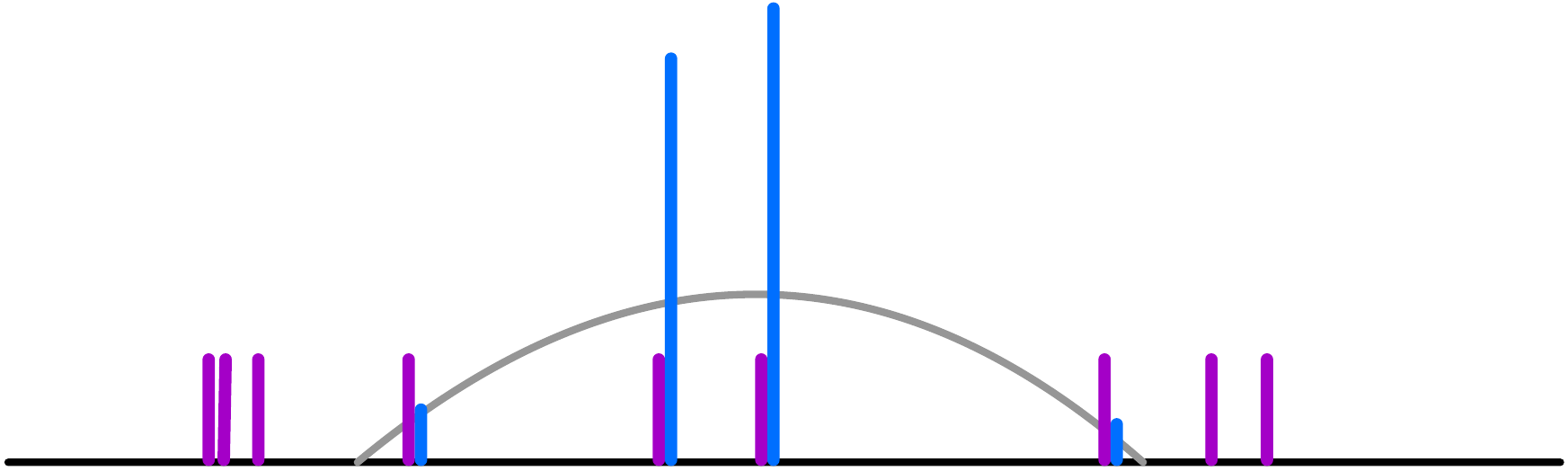}
\caption{A drawing of the pumping lemma when $X=\R$, with map $\phi$ in gray, measure $\mu$ in purple, and measure $f(\mu)$ in blue.}
\label{fig:pumping}
\end{figure}

\begin{lemma}[Pumping lemma]
\label{lem:homotopy2}
Let $X$ be a metric space, and let the open set $U\subseteq X$ have finite diameter.
Suppose $A \subseteq \cPfin{X}$ has $\MCP(p,U)$ for some $0<p<1$, and let $\phi\colon X \to [0,1]$ be an $L$-Lipschitz map with $\phi^{-1}(0)=U^C$.
Then $f\colon A \to M_U\subseteq\cPfin{X}$ defined as 
\[
f\Big(\sum_i a_i \delta_{x_i}\Big) = \sum_i \left(\frac{a_i \phi(x_i)}{\sum_j a_j \phi(x_j)}\right) \delta_{x_i}
\]
is continuous.
If $A$ is furthermore $U$-pumping convex, then the homotopy $H \colon A \times [0, 1] \to A$ defined by $H(\mu, t) = (1 - t)\mu + t f(\mu)$ is well-defined and continuous.
\end{lemma}

\begin{proof}
We will prove the continuity of $f$ at an arbitrary point $\mu=\sum_ia_i\delta_{x_i}\in A$.
Define $a:=\sum_i a_i \phi(x_i)>0$, which is positive since $A$ has $\MCP(p,U)$ and since $\phi(x)>0$ for all $x\in U$.
Since $U$ has finite diameter, there exists some constant $D<\infty$ such that $\diam(U)<D$.

Let $\varepsilon >0$.
Choose $0<\delta\le\min(\frac{a\varepsilon}{2},\frac{a}{L})$ sufficently small so that $(1-\frac{a-L\,\delta}{a+L\,\delta})D \le \frac{\varepsilon}{2}$.
Suppose $\nu:=\sum_j b_j\delta_{y_j}\in A$ with $d_W(\mu,\nu) < \delta$.
This means there exist $q_{i,j}\ge 0$ with $\sum_j q_{i,j}=a_i$, with $\sum_i q_{i,j}=b_j$, and with $\sum_{i,j}q_{i,j}d_X(x_i,y_j)<\delta$.
We will bound the 1-Wasserstein distance between the measures $f(\mu)$ and $f(\nu)$ by describing how to transport only part of the mass, and then using the partial coupling lemma.

Define $b:=\sum_j b_j \phi(y_j)>0$.
We define a \emph{partial} transport plan $\{c_{i,j}\}$ between $f(\mu)$ and $f(\nu)$ via 
\[c_{i,j} = q_{i,j}\min\left(\frac{\phi(x_i)}{a},\frac{\phi(y_j)}{b}\right).\]
We have
\begin{align*}
\sum_jc_{i,j} &= \sum_jq_{i,j}\min\left(\frac{\phi(x_i)}{a},\frac{\phi(y_j)}{b}\right) \le \sum_jq_{i,j}\frac{\phi(x_i)}{a} = \frac{a_i\phi(x_i)}{a} \quad\text{and} \\
\sum_i c_{i,j} &= \sum_i q_{i,j}\min\left(\frac{\phi(x_i)}{a},\frac{\phi(y_j)}{b}\right) \le \sum_{i}q_{i,j}\frac{\phi(y_i)}{b} = \frac{b_j \phi(y_j)}{b}.
\end{align*}
This shows that $\{c_{i,j}\}$ is indeed a partial transport plan from $f(\mu)=\sum_i \frac{a_i \phi(x_i)}{a} \delta_{x_i}$ to $f(\nu)=\sum_j \frac{b_j \phi(y_j)}{b} \delta_{y_j}$.
Note that the cost of this partial transport plan is small, namely
\[
\sum_{i,j}c_{i,j}d(x_i,y_j)
= \sum_{i,j}q_{i,j}\min\left(\frac{\phi(x_i)}{a},\frac{\phi(y_j)}{b}\right)d(x_i,y_j)
\le \frac{1}{a}\sum_{i,j}q_{i,j}d(x_i,y_j)
< \frac{\delta}{a}.
\]

To apply the partial coupling lemma, we also need a lower bound on the amount of mass $\sum_{i,j}c_{i,j}$ we have transported.
For this, we will need the bound
\begin{align*}
|a - b| &= \left|\sum_i a_i\phi(x_i)-\sum_j b_j\phi(y_j)\right| 
= \left|\sum_{i,j}q_{i,j}(\phi(x_i)-\phi(y_j))\right| \\
&\le \sum_{i,j}q_{i,j}|\phi(x_i)-\phi(y_j)| 
\le L \sum_{i,j}q_{i,j}d_X(x_i,y_j) < L\,\delta,
\end{align*}
where the second-to-last step is since $\phi$ is $L$-Lipschitz.
Now, the amount of mass transported by our partial transport plan is at least
\begin{align*}
\sum_{i,j}c_{i,j} 
&= \sum_{i,j} q_{i,j}\min\left(\frac{\phi(x_i)}{a},\frac{\phi(y_j)}{b}\right) \\
&\ge \frac{1}{\max(a,b)} \sum_{i,j} q_{i,j}\min\left(\phi(x_i),\phi(y_j)\right) \\
&\ge \frac{1}{a+L\,\delta} \sum_{i,j} q_{i,j}\min\left(\phi(x_i),\phi(y_j)\right) \\
&\ge \frac{1}{a+L\,\delta} \sum_{i,j} q_{i,j}\left(\phi(x_i)-L\,d_X(x_i,y_j)\right) &&\text{since }\phi\text{ is $L$-Lipschitz}\\
&\ge \frac{1}{a+L\,\delta} \left(\sum_{i,j} q_{i,j}\phi(x_i) - L \sum_{i,j} q_{i,j}d_X(x_i,y_j)\right) \\
&\ge \frac{1}{a+L\,\delta} \left(\sum_{i} a_i \phi(x_i) - L\,\delta\right) \\
&= \frac{a-L\,\delta}{a+L\,\delta}.
\end{align*}
By definition, both $f(\mu)$ and $f(\nu)$ have their supports in $U$, which has diameter at most $D$.
Hence we apply the partial coupling lemma (Lemma~\ref{lem:partial-transport}) to get
\begin{align*}
d_W(f(\mu), f(\nu)) &\le \sum_{i,j}c_{i,j}d_X(x_i,y_j) + \left(1-\sum_{i,j}c_{i,j}\right)\diam\left(\supp(f(\mu))\cup\supp(f(\nu))\right) \\
&\le \tfrac{\delta}{a} + (1-\tfrac{a-L\,\delta}{a+L\,\delta}) D \\
&\le \tfrac{\varepsilon}{2} + \tfrac{\varepsilon}{2} \\
&= \varepsilon.
\end{align*}
Hence $f$ is continuous.

We note that for any measure $\mu\in A$ we have $\supp(f(\mu))\subseteq \supp(\mu)\cap U$ since $\phi^{-1}(0)=U^C$.
Therefore, if $A$ is $U$-pumping convex, then the homotopy $H \colon A \times [0, 1] \to A$ defined by $H(\mu, t) = (1 - t)\mu + t f(\mu)$ is well-defined since $\supp(f(\mu))\subseteq \supp(\mu)\cap U$.
This homotopy is also continuous by Lemma~\ref{lem:homotopy1}.
\end{proof}

Let $\cU$ be a uniformly bounded open cover of a metric space $X$. 
We now define the open cover $\wM_\cU$ of $\cVm(\cU)$, before showing that it is good up to level $n$.

Choose $0<p<1$.
Fix $U\in \cU$.
We will define an open neighborhood $\wM_U$ of $M_U$ in $\cVm(\cU)$ in the following inductive manner.
We construct an increasing sequence
\[
M_U = Q_0 \subseteq N_1 \subseteq Q_1 \subseteq N_2 \subseteq Q_2 \subseteq \ldots
\]
so that each element of the sequence has the mass concentration property $\MCP(p,U)$, 
each $N_k$ is open, and 
each $Q_k$ is $U$-pumping convex.
Consequently $\wM_U := \cup_k Q_k = \cup_k N_k$ will be open, be $U$-pumping convex, and have the mass concentration property $\MCP(p,U)$ (since if $\mu\in \wM_U$, then we have $\mu\in Q_k$ for some $k$, and hence $\mu(U)>p$).

Let us now inductively define the sets $Q_k$ and $N_k$.
We set $Q_0:=M_U$, which has $\MCP(p,U)$ and which is $U$-pumping convex.
Fix $k\in \mathbb{N}$.
By induction, each $Q_{k-1}$ has $\MCP(p,U)$, meaning any $\mu \in Q_{k-1}$ satisfies $\mu(U)>p$.
Note $d(\supp(\mu)\cap U,U^C)>0$ since $\supp(\mu)$ is finite and $U$ is open.
If $\nu \in \cVm(\cU)$ is a measure with $\nu(U)\le p$, then the cost of a transport plan from $\mu$ to $\nu$ must be at least as large as the amount of mass from $\mu$ that needs to be moved outside of $U$ times the distance from $\supp(\mu)\cap U$ to the complement of $U$, which is $(\mu(U)-p)\,d(\supp(\mu)\cap U,U^C)$.
Therefore, if we choose $r_\mu>0$ so that $r_{\mu} < (\mu(U)-p)\,d(\supp(\mu)\cap U,U^C)$, then any measure with Wasserstein distance less than $r_\mu$ from $\mu$ has mass totaling more than $\mu(U)-(\mu(U)-p) = p$ inside of $U$.
In other words, the $r_{\mu}$-neighborhood of $\mu$, denoted by $B_{\cVm(\cU)}(\mu; r_{\mu})$, also has $\MCP(p,U)$.
%Check if argument works for infinitely supported case---compact set contains most of measure.
Define $N_k = \bigcup_{\mu \in Q_{k-1}} B_{\cVm(\cU)}(\mu; r_{\mu})$.
Note that $N_k$ is open and has $\MCP(p,U)$.
Define
\[ Q_k = \{(1-t)\mu + t \nu~\mid~\mu\in N_k,\ \nu \in \cPfin{X} \text{ with } \supp(\nu)\subseteq\supp(\mu)\cap U,\ t\in [0,1]\} \subseteq \cVm(\cU). \]
It follows that $Q_k$ has $\MCP(p,U)$ and is $U$-pumping convex.
Then, as we mentioned, we define $\wM_U = \cup_k Q_k = \cup_k N_k$, which is open, is $U$-pumping convex, and has the mass concentration property $\MCP(p,U)$.
Define $\wM_\cU := \{\wM_U\}_{U \in \cU}$.

\begin{proposition}
\label{prop:almostGoodCover}
Let $\cU$ be a uniformly bounded open cover of a metric space $X$, let $n\in \mathbb{N}$ with $n\ge 2$, and fix $1-1/n<p<1$.
If $\wM_\cU$ is constructed using this choice of $p$, then $\wM_\cU$ is an open cover of $\cVm(\cU)$ that is good up to level $n$.
Furthermore, the $n$-skeleta of the nerve complexes $\cN(\wM_\cU)$ and $\cN(\cU)$ coincide.
\end{proposition}

\begin{proof}
Recall that each set $\wM_U$ is open and has $\MCP(p,U)$ by construction.

Note $1-1/n<p<1$ implies that $0<n(1-p)<1$.
We claim that for any $k\le n$ and $U_1, U_2, \ldots, U_k \in \cU$,
\begin{equation}
\label{eq:wmcp-intersect}
\wM_{U_1}\cap \ldots \cap \wM_{U_k} \text{ has } \MCP\Big(1-n(1-p), U_1 \cap \ldots \cap U_k\Big).    
\end{equation}
Indeed, choose any $\mu \in \wM_{U_1} \cap \ldots \cap \wM_{U_k}$.
Since $\mu(U_i^C) < 1-p$ for each $i$, we have
\[\mu((U_1 \cap \ldots \cap U_k)^C) = \mu(U_1^C \cup \ldots \cup U_k^C) < k(1-p)\le n(1-p).\]

Next, we show that $\wM_\cU$ is a good cover up to level $n$.
For $k\le n$ and for arbitrary sets $U_1,\ldots,U_k\in\cU$, assume that $\wM_{U_1} \cap \ldots \cap \wM_{U_k}$ is nonempty.
Since the finite intersection of open sets $U_1 \cap \ldots \cap U_k$ is open, we can choose a 1-Lipschitz function $\phi \colon X \to [0,1]$ for which $\phi^{-1}(0)=(U_1 \cap \ldots \cap U_k)^C$.
By~\eqref{eq:wmcp-intersect} and Lemma~\ref{lem:homotopy2}, the map $f\colon \wM_{U_1}\cap \ldots \cap \wM_{U_k} \to M_{U_1}\cap \ldots \cap M_{U_k}$ defined as 
\[
f\Big(\sum_i a_i \delta_{x_i}\Big) = \sum_i \left(\frac{a_i \phi(x_i)}{\sum_j a_j \phi(x_j)}\right) \delta_{x_i}
\]
is continuous.
Each $\wM_{U_i}$ is $U_i$-pumping convex and therefore $(U_1 \cap \ldots \cap U_k)$-pumping convex, and so it follows that $\wM_{U_1}\cap \ldots \cap \wM_{U_k}$ is also $(U_1 \cap \ldots \cap U_k)$-pumping convex.
Hence Lemma~\ref{lem:homotopy2} furthermore implies that the homotopy $H \colon \wM_{U_1}\cap \ldots \cap \wM_{U_k} \times [0, 1] \to \cVm(\cU)$ defined by $H(\mu, t) = (1 - t)\mu + t f(\mu)$ is well-defined and continuous.
Therefore, the identity on $\wM_{U_1}\cap \ldots \cap \wM_{U_k}$ is homotopic to the map~$f$, whose image lies in the contractible space $M_{U_1}\cap \ldots \cap M_{U_k}=M_{U_1\cap\ldots\cap U_k}$.
So $\wM_{U_1} \cap \ldots \cap \wM_{U_k}$ is contractible.
This shows that $\wM_\cU$ is a good cover up to level $n$.

Lastly, we show that the $n$-skeleta of the nerve complexes $\cN(\wM_\cU)$ and $\cN(\cU)$ coincide.
Let $k\le n$ and consider arbitrary sets $U_1,\ldots,U_k\in\cU$.
If $x\in U_1 \cap \ldots \cap U_k$ then the Dirac measure at $x$ is contained in $\wM_{U_1} \cap \ldots \cap \wM_{U_k}$.
On the other hand, if $ U_1 \cap \ldots \cap U_k=\emptyset$ then $\wM_{U_1}\cap \ldots \cap \wM_{U_k}$ contains no measure by~\eqref{eq:wmcp-intersect} and by the definition of the mass concentration property, as $1-n(1-p) > 0$.
\end{proof}

\section{Vietoris thickenings and complexes have isomorphic homotopy groups}
\label{sec:main}

We are now prepared to use the open cover from Section~\ref{sec:cover-open}, which is good up to level $n$, to prove that $\cV(\cU)$ and $\cVm(\cU)$ have isomorphic homotopy groups.

\begin{theorem}
\label{thm:main}
If $\cU$ is a uniformly-bounded open cover of a separable metric space $X$, then the Vietoris thickening $\cVm(\cU)$ and the Vietoris complex $\cV(\cU)$ have isomorphic homotopy groups $\pi_k(\cVm(\cU))\cong \pi_k(\cV(\cU))$ in all dimensions $k\ge 0$.
\end{theorem}

\begin{proof}

The organization of the proof is as follows.
Fix an arbitrary integer $n\ge0$.
We can construct a cover $\wM_\cU$ of $\cVm(\cU)$ that is good up to level $n$, and we know from Proposition~\ref{prop:almostGoodCover} that this cover is open.
In Lemma~\ref{lem:nagorko} which follows, we will apply a result from Nagorko~\cite{nagorko2007carrier} to conclude that the first $n-1$ homotopy groups of $\cVm(\cU)$ and $\cN(\wM_\cU)$ are isomorphic, i.e.\ $\pi_k(\cVm(\cU)) \cong \pi_k(\cN(\wM_\cU))$ for all $0\le k\le n-1$.
Since the $n$-skeleta of $\cN(\wM_\cU)$ and $\cN(\cU)$ coincide by Proposition~\ref{prop:almostGoodCover}, it follows from cellular approximation that the first $n-1$ homotopy groups of $\cN(\wM_\cU)$ and $\cN(\cU)$ are isomorphic, i.e.\ $\pi_k(\cN(\wM_\cU)) \cong \pi_k(\cN(\cU))$ for all $0\le k\le n-1$.
Finally, we conclude using the fact that
$\cN(\cU) \simeq \cV(\cU)$ by Dowker Duality~\cite{dowker1952topology,zeeman1962dihomology,virk2021rips}.
Stringing these facts together, we get that
\[ \pi_k(\cVm(\cU)) \cong \pi_k(\cN(\wM_\cU)) \cong \pi_k(\cN(\cU)) \cong \pi_k(\cV(\cU)) \quad\text{for }0\le k\le n-1.\]
Since the integer $n$ can be made arbitrarily large (although with different covers $\wM_\cU$ for different values of $n$), we get that $\cVm(\cU)$ and $\cV(\cU)$ have isomorphic homotopy groups $\pi_k$ for all integers $k\ge 0$.
\end{proof}

It remains to prove Lemma~\ref{lem:nagorko}, which states that first $n-1$ homotopy groups of $\cVm(\cU)$ and $\cN(\wM_\cU)$ are isomorphic.

\begin{lemma}
\label{lem:nagorko}
Let $\cU$ be a uniformly-bounded open cover of a separable metric space $X$, let $n\ge 2$, fix $1-1/n<p<1$, and construct $\wM_\cU$ using this choice of $p$.
Then $\pi_k(\cVm(\cU)) \cong \pi_k(\cN(\wM_\cU))$ for all $0\le k\le n-1$.
\end{lemma}

\begin{proof}
We will apply Theorem~3.4 of~\cite{nagorko2007carrier} by Nagorko, which says that if $\cF$ is an open cover of a separable space $Y$, weakly regular for the class of at most $n$-dimensional spaces, then each canonical map $Y \to \cN(\cF)$ (induced by a partition of unity subordinated to this cover) produces isomorphisms on homotopy groups of dimensions less than $n$.
We refer the reader to~\cite{nagorko2007carrier} for the definition of \emph{weakly regular for the class of at most $n$-dimensional spaces}, since we will not need it here.
Indeed, two paragraphs after Theorem~3.4, Nagorko states that the Excision Theorem implies that an open cover $\cF$ of a locally $(n-1)$-connected space is weakly regular for the class of at most $n$-dimensional spaces if and only if each nonempty intersection of a collection $\cA\subseteq\cF$ is $(n-|\cA|)$-connected.
This is the condition we will verify.

We apply~\cite[Theorem~3.4]{nagorko2007carrier} to the map $\cVm(\cU)\to\cN(\wM_\cU)$.
First, note that $\cVm(\cU)$ is separable since $X$ is separable; we prove this in Lemma~\ref{lem:separable} which follows.
To see that $\cVm(\cU)$ is locally $(n-1)$-connected, note that it is furthermore locally contractible by Theorem~\ref{thm:locally-contractible} in Section~\ref{sec:locally-contractible} since the cover $\cU$ is uniformly bounded.
This is the only place where our assumption that the cover $\cU$ is uniformly bounded is used.
Finally, to see that the nonempty intersection of a collection $\cA\subseteq\cF$ is $(n-|\cA|)$-connected, note this follows from Proposition~\ref{prop:almostGoodCover}, which says that intersections up to level $n$ are empty or contractible.
\end{proof}

Let us prove Lemma~\ref{lem:separable} about separability that we used above.

\begin{lemma}
\label{lem:separable}
If $\cU$ is a cover of a separable metric space $X$, then $\cVm(\cU)$ is separable.
\end{lemma}

\begin{proof}
Since $X$ is separable, consider a countable dense subset $Z\subseteq X$.
Now define the subset \[ \cVm_{Z,\Q}(\cU):=\left\{\sum_i a_i\delta_{x_i}\in \cVm(\cU)~|~a_i\in\Q \text{ and } x_i\in Z \text{ for all }i\right\}. \]
Since $Z$ is countable, since $\Q$ is countable, and since each measure in $\cVm_{Z,\Q}(\cU)$ has finite support, it follows that $\cVm_{Z,\Q}(\cU)$ is countable.
The fact that $\cVm_{Z,\Q}(\cU)$ is dense in $\cVm(\cU)$ follows since $Z$ is dense in $X$ and $\Q$ is dense in $\R$.
\end{proof}

Finally, in the following section, we explain why $\cVm(\cU)$ is locally contractible when $\cU$ is a uniformly-bounded open cover.

\section{Vietoris metric thickenings are locally contractible}
\label{sec:locally-contractible}

Five years ago, the first two authors tried to show that the metric thickening $\vrm{X}{r}$ is locally contractible to include as a result in~\cite{AAF}, but they did not succeed.
The following theorem answers this matter in the affirmative, and more generally holds for any uniformly-bounded open cover.

\begin{theorem}
\label{thm:locally-contractible}
If $\cU$ is a uniformly-bounded open cover of a metric space $X$, then $\cVm(\cU)$ is locally contractible.
\end{theorem}

\begin{proof}
Let $\mu\in\cVm(\cU)$ and let $s>0$ be arbitrary.
It suffices to show that there exists some $0<s'<\frac{s}{2}$ such that $B_{\cVm(\cU)}(\mu;s')$ is contractible in $B_{\cVm(\cU)}(\mu;s)$.
Let $\supp(\mu)=\{y_0,\ldots,y_k\}\subseteq U\in\cU$.
Choose $0<\varepsilon<\frac{1}{2}d_X(\supp(\mu),U^C)$.
%\note{I changed $\frac{1}{2}$ in the definition of $\varepsilon$ to $\frac{1}{2}$, since I think that is all that is needed.}
Let $Y_1=\cup_i B_X(y_i;\varepsilon)$ and let $Y_2=\cup_i B_X(y_i;2\varepsilon)$.
These sets are open in $X$ and are shown in Figure~\ref{fig:Y1Y2}.
The choice of $\varepsilon$ implies $Y_2\subseteq U$, and hence $M_{Y_2}\subseteq\cVm(\cU)$.
Fix $0<p<1$ large enough so that $(1-p)D<\frac{s}{2}$, where $D<\infty$ is the uniform bound such that $\diam(U)<D$ for all $U\in \cU$.
Choose $s' < \frac{s}{2}$ so that $B_{\cVm(\cU)}(\mu;s')$ has $\MCP(p,Y_1)$.\footnote{One choice that suffices is to pick $s' < \frac{s}{2}$ to furthermore satisfy $s'<(1-p)\varepsilon$, since then $\nu\in B_{\cVm(\cU)}(\mu;s')$ implies that some transport plan between $\mu$ and $\nu$ has cost less than $(1-p)\varepsilon$, which means that less than $1-p$ of the mass in $\nu$ can be outside of $Y_1=\cup_i B_X(y_i;\varepsilon)$.}
Fix a $\frac{1}{\varepsilon}$-Lipschitz function $\phi\colon X\to[0,1]$ with $\phi^{-1}(1)=Y_1$ and $\phi^{-1}(0)=Y_2^C$, which is possible\footnote{For example, define $\phi(x)=\frac{d(x,Y_2^C)}{d(x,Y_1)+d(x,Y_2^C)}$.}
since $d(Y_1,Y_2^C)\ge\varepsilon$.
Define $f\colon B_{\cVm(\cU)}(\mu;s')\to M_{Y_2}\subseteq \cVm(\cU)$ by
\[
f\Big(\sum_i a_i \delta_{x_i}\Big) = \sum_i \left(\frac{a_i \phi(x_i)}{\sum_j a_j \phi(x_j)}\right) \delta_{x_i}
\]
This function is continuous by Lemma~\ref{lem:homotopy2}, and it maps into $M_{Y_2}$ since $\phi^{-1}(0)=Y_2^C$.
Also, note that $\phi(x_i)\le 1$ for all $x_i$ and $\phi(x_i)=1$ if $x_i\in Y_1$, which together imply that $f(\nu)(x_i)\ge \nu(x_i)$ for $x_i\in Y_1$.

\begin{figure}[h]
\centering
\includegraphics[width=4in]{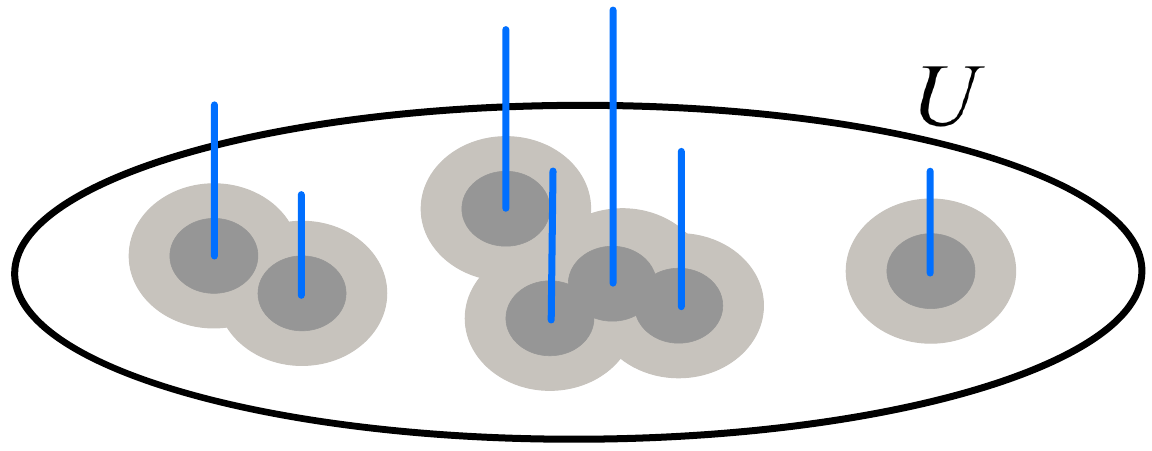}
\caption{A drawing for the proof of Theorem~\ref{thm:locally-contractible} when $X=\R^2$.
Measure $\mu$ is blue, $U$ is black, $Y_2$ is light gray, and $Y_1$ is dark gray.}
\label{fig:Y1Y2}
\end{figure}

Define a homotopy $H\colon B_{\cVm(\cU)}(\mu;s')\times [0,1] \to B_{\cVm(\cU)}(\mu;s)$ by
\[ H(\nu,t)=(1-t)\nu+tf(\nu). \]
By Lemma~\ref{lem:homotopy1}, the homotopy $H$ is continuous so long as it is well-defined, which we confirm now.
First note that $H(\nu,t)\in M_{Y_2}\subseteq \cVm(\cU)$.
For $H$ to be well-defined, we must also show that
$H(\nu,t)\in B_{\cVm(\cU)}(\mu;s)$, i.e.\ that
$d_W(H(\nu,t),\mu) < s$, for all $\nu\in B_{\cVm(\cU)}(\mu;s')$ and $t\in [0,1]$.
Since $B_{\cVm(\cU)}(\mu;s')$ has $\MCP(p,Y_1)$, we know that $\nu(Y_1) > p$.
Furthermore, $f(\nu)(x_i)\ge \nu(x_i)$ for $x_i\in Y_1$ implies that $H(\nu,t)(x_i)\ge \nu(x_i)$ for all $x_i\in Y_1$ and $t\in [0,1]$.
We claim that $d_W(H(\nu,t),\nu)\le (1-p)D$.
Indeed, a transport plan from $H(\nu,t)$ to $\nu$ can leave mass $\nu(x_i)$ at each point $x_i\in Y_1$, and hence only needs to move mass totaling $\nu(Y_1^C) < 1-p$ some distance at most $D$ (since $\supp(H(\nu,t))\subseteq\supp(\nu)\subseteq V$ for some $V\in \cU$, which means $\diam(V)\le D$).
Therefore,
\begin{align*}
d_W(H(\nu,t),\mu)
&\le d_W(H(\nu,t),\nu)+d_W(\nu,\mu) \\
&\le (1-p)D + s' \\
&<\tfrac{s}{2}+\tfrac{s}{2} \\
&= s.
\end{align*}

Finally, we show that the inclusion map of $\im(H(\cdot,1))=\im(f)$ into $B_{\cVm(\cU)}(\mu;s)$ is nullhomotopic to the constant map to $\mu$.
Indeed, define $F\colon \im(f)\times [0,1]\to B_{\cVm(\cU)}(\mu;s)$ by
\[ F(\omega,t)=(1-t)\omega+t\mu. \]
Note 
$F(\omega,t)\in M_{Y_2}\subseteq \cVm(\cU)$ and
\[ d_W(F(\omega,t),\mu) \le (1-t)d_W(\omega,\mu) + t d_W(\mu,\mu) < (1-t)s \le s, \]
where the first inequality above follows from~\cite[Lemma~4.2]{AMMW}, or more generally, from~\cite[Theorem 4.8]{villani2008optimal}.
So $F$ is well-defined, and continuous by Lemma~\ref{lem:homotopy1}.

Together, the homotopies $H$ and $F$ show that $B_{\cVm(\cU)}(\mu;s')$ is contractible in $B_{\cVm(\cU)}(\mu;s)$, as desired.
\end{proof}

\begin{figure}[h]
\centering
\includegraphics[width=3in]{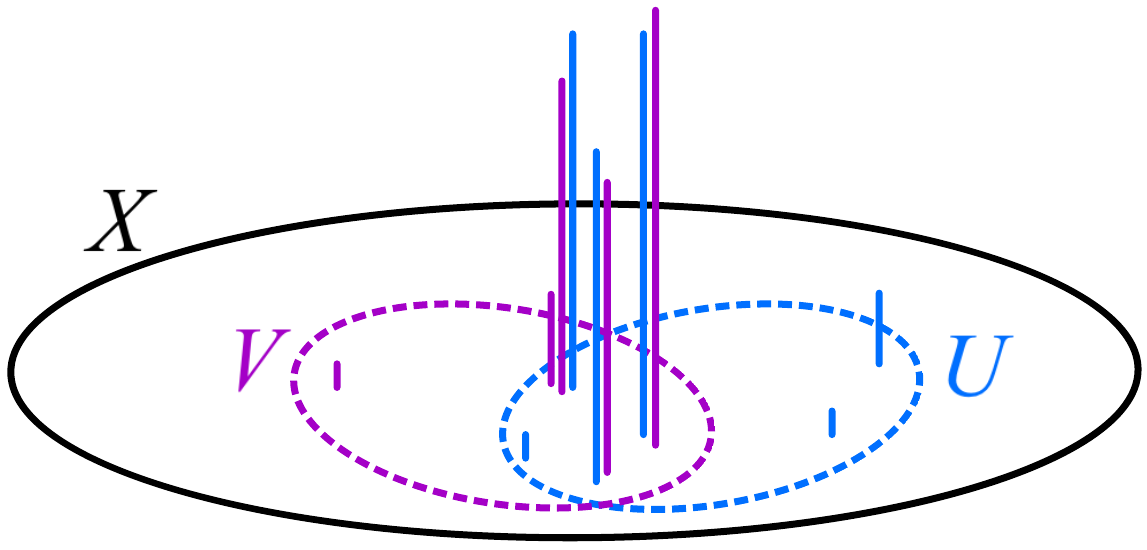}
\caption{A measure $\mu$ with $\supp(\mu)\in U\in \cU$ in blue, and a measure $\nu\in B_{\cVm(\cU)}(\mu;s')$ with $\supp(\nu)\in V\in \cU$.
We cannot linearly homotope from $\mu$ directly to $\nu$ since there may be no set in $\cU$ containing $\supp(\mu)\cup\supp(\nu)$.
}
\label{fig:LinearNotDefined}
\end{figure}

\begin{remark}
The reason why we need two homotopies $H$ and $F$, and cannot linearly homotope from $\mu$ directly to $\nu$, is that $\nu\in B_{\cVm(\cU)}(\mu;s')$ need not imply that there is any set in $\cU$
that contains $\supp(\mu)\cup\supp(\nu)$.
Indeed, see Figure~\ref{fig:LinearNotDefined}.
\end{remark}

\section{Conclusion}
\label{sec:conclusion}

We end with some open questions.
By choosing the cover $\cU$ appropriately, all of the questions in this section have consequences for Vietoris--Rips and \v{C}ech complexes and thickenings.

%From our proof relying on~\cite{nagorko2007carrier}, for any integer $n$ we produce a map between $\cV(\cU)$ and $\cVm(\cU)$ that induces isomorphisms on homotopy groups $\pi_k$ for all $0\le k\le n-1$.
%Even though $n$ is arbitrary, we do not obtain a single map that works for all $n$.
A continuous map between topological spaces is a \emph{weak equivalence} if it induces an isomorphism on all homotopy groups.
Even though we show that the Vietoris thickening $\cVm(\cU)$ and the Vietoris complex $\cV(\cU)$ have isomorphic homotopy groups in all dimensions, we do not find a map inducing these isomorphisms.

\begin{question}
\label{ques:homotopy-equiv}
If $\cU$ is a uniformly bounded open cover of a separable metric space $X$, then are the Vietoris thickening $\cVm(\cU)$ and the Vietoris complex $\cV(\cU)$ weakly homotopy equivalent?
Are they homotopy equivalent?
Are more hypotheses on $X$ or $\cU$ needed, or do fewer hypotheses suffice?
\end{question}

Theorem~\ref{thm:main} states that $\cV(\cU)$ and $\cVm(\cU)$ have isomorphic homotopy groups, and Question~\ref{ques:homotopy-equiv} asks if they are weakly homotopy equivalent or homotopy equivalent.
We now ask if a particular natural map realizes these relationships.
Recall from the end of Section~\ref{sec:prelim} that there is a natural bijection $f\colon \cV(\cU)\to\cVm(\cU)$ from the Vietoris complex to the Vietoris thickening, obtained by assigning a point in the geometric realization of $\cV(\cU)$ associated to the simplex $\{x_0,\ldots,x_k\}$ with barycentric coordinates $(a_0,\ldots,a_k)$ to the measure $\sum_{i=0}^k a_i\delta_{x_i}\in\cVm(\cU)$.
By~\cite[Proposition~6.1]{AAF} the map $f$ is continuous, but its inverse $f^{-1}$ may be discontinuous when $X$ is infinite.

\begin{question}
If $\cU$ is a uniformly bounded open cover of a separable metric space $X$, then does the natural map $f\colon \cV(\cU)\to\cVm(\cU)$ from the Vietoris complex to the Vietoris thickening induce an isomorphism on all homotopy groups, and furthermore is $f$ a homotopy equivalence?
\end{question}

Whitehead's theorem states that a weak homotopy equivalencee between CW complexes is a homotopy equivalence. 
Therefore, an affirmative answer to either of the following questions would be useful towards showing that if $\cVm(\cU)$ and $\cV(\cU)$ are weakly homotopy equivalent, then they are also homotopy equivalent, by using Whitehead's theorem.

\begin{question}
If $\cU$ is a uniformly bounded open cover of a separable metric space $X$, then is $\cV(\cU)$ an absolute neighborhood retract (ANR)~\cite{borsuk1932klasse,nhu1989probability}?
Every ANR has the homotopy type of a CW complex~\cite[Theorem~5.2.1]{fritsch1990cellular}.
\end{question}
% Connection to CW complexes.
%\note{For homotopy equivalence, need only Nagorko type result for covers of complexes, and that metric thickening is ANR; keep working towards this.}

\begin{question}
If $X$ is a metrizable CW complex and if $r>0$, then are $\vrm{X}{r}$ and $\cechm{X}{r}$ homotopy equivalent to CW complexes?
\end{question}

In this paper we have focused on homotopy groups, but similar questions can be asked about how the homology groups of $\cVm(\cU)$ and $\cV(\cU)$ relate.
Good covers up to level $n$ could potentially be included in a Mayer--Vietoris spectral sequence argument to show that homology groups are isomorphic up to dimension $n-1$; see~\cite[\S 15]{bott1982differential} and~\cite{brown2012cohomology,cardona,dugger2004topological}.
%\url{https://www.albany.edu/~rc782885/talk/2018-10-october-ami/2018-10-october-ami.pdf}
%\url{https://uregina.ca/~franklam/Math527/Math527_0325.pdf}

\begin{question}
If $\cU$ is a uniformly bounded open cover of a separable metric space $X$, then do the Vietoris thickening $\cVm(\cU)$ and the Vietoris complex $\cV(\cU)$ have the same homology groups?
\end{question}

\noindent This question has recently been answered in the affirmative by Patrick Gillespie, and moreover the assumption of separability is not needed; we refer the reader to~\cite{gillespie2022homological}.

In addition to homology, one can ask about \emph{persistent} homology.
For $X$ a totally bounded metric space, it is shown in~\cite{AMMW,MoyMasters} that the Vietoris--Rips simplicial complex filtration $\vr{X}{-}$ and the Vietoris--Rips metric thickening filtration $\vrm{X}{-}$ are $\varepsilon$-interleaved for any $\varepsilon>0$.
The same is true for the \v{C}ech filtrations $\cech{X}{-}$ and $\cechm{X}{-}$.
It follows that $\vr{X}{-}$ and $\vrm{X}{-}$ have the same persistent homology and persistent homotopy groups, and similarly for $\cech{X}{-}$ and $\cechm{X}{-}$.
To be more explicit, in the Vietoris--Rips case, this means that for any $r<r'$, the inclusions $\vr{X}{r}\hookrightarrow \vr{X}{r'}$ and $\vrm{X}{r}\hookrightarrow \vrm{X}{r'}$ induce isomorphic images on homology $H_k$ and on homotopy groups $\pi_k$ for all $k\ge 0$, i.e.,
\begin{align*}
\im\Bigl(H_k\bigl(\vr{X}{r}\bigr)\to H_k\bigl(\vr{X}{r'}\bigr)\Bigr)&\cong\im\Bigl(H_k\bigl(\vrm{X}{r}\bigr)\to H_k\bigl(\vrm{X}{r'}\bigr)\Bigr)\quad\text{and} \\
\im\Bigl(\pi_k\bigl(\vr{X}{r}\bigr)\to \pi_k\bigl(\vr{X}{r'}\bigr)\Bigr)&\cong\im\Bigl(\pi_k\bigl(\vrm{X}{r}\bigr)\to \pi_k\bigl(\vrm{X}{r'}\bigr)\Bigr)\quad\text{for any }r<r'.
\end{align*}

\begin{question}
Do our results hold for variants of the metric thickening that allow for infinitely supported measures (see~\cite{AMMW}, for example)?
\end{question}
% Watch out for first bullet in inductive definition in tilde MU sets.}

\section*{Acknowledgements}

The authors would like to thank Robert Cardona for helpful conversations.
This research was supported through the program ``Research in Pairs" by the Mathematisches Forschungsinstitut Oberwolfach in 2019.
The second author was supported by NSF grant DMS 1855591, NSF CAREER grant DMS 2042428, and a Sloan Research Fellowship.
The third author was supported by Slovenian Research Agency grants No.\ N1-0114 and P1-0292.
The first and third authors would like to thank the Institute of Science and Technology Austria (ISTA) for hosting research visits.

\bibliographystyle{plain}
\bibliography{VietorisThickeningsAndComplexesHaveIsomorphicHomotopyGroups.bib}

\end{document}